\documentclass[reqno]{amsart}
\usepackage{amsmath,amssymb,cite}
\usepackage[mathscr]{euscript}
\usepackage{mathtools}
\usepackage[scr=boondox,  
            cal=esstix]   
           {mathalpha}
\usepackage{xcolor}
\usepackage{hyperref}
\usepackage{dsfont}
\usepackage{tikz-cd}
\usepackage{comment}
\usepackage[normalem]{ulem}
\usepackage{bbold,dsfont}

\newcommand{\overbar}[1]{\mkern 2.0mu\overline{\mkern-4mu#1\mkern-2.0mu}\mkern 2.0mu}

\newcommand{\overtilde}[1]{\mkern 2.0mu\widetilde{\mkern-3mu#1\mkern-2.0mu}\mkern 2.0mu}

\definecolor{bblue}{rgb}{.2,0.2,.8}

\theoremstyle{plain}
\newtheorem{theorem}{Theorem}[section]
\newtheorem{proposition}[theorem]{Proposition}
\newtheorem{lemma}[theorem]{Lemma}
\newtheorem{corollary}[theorem]{Corollary}

\theoremstyle{definition}
\newtheorem{definition}[theorem]{Definition}
\newtheorem{assumption}[theorem]{Assumption}

\theoremstyle{remark}
\newtheorem{remark}[theorem]{Remark}

\numberwithin{equation}{section}
\numberwithin{theorem}{section}

\def\be{\begin{equation}}
\def\ee{\end{equation}}
\def\bp{\begin{pmatrix}}
\def\ep{\end{pmatrix}}
\def\bea{\begin{eqnarray}}
\def\eea{\end{eqnarray}}

\def\\{\par\medskip}

\let\0=\noindent

\newcommand*{\defeq}{\mathrel{\vcenter{\baselineskip0.5ex \lineskiplimit0pt
                     \hbox{\scriptsize.}\hbox{\scriptsize.}}}%
                     =}

\renewcommand{\epsilon}{\varepsilon}

\DeclareMathOperator{\Aut}{Aut}
\DeclareMathOperator{\Ann}{Ann}
\DeclareMathOperator{\Stem}{Stem}
\DeclareMathOperator{\rk}{rk}

\title[A classification of $\mathbb{Z}_p$-braces]{A classification of module braces   over the ring of $\mathbf{p}$-adic integers}

\author{Riccardo Aragona}
\address{\noindent Riccardo Aragona \hfill\break\indent 
 DISIM, Universit\`a dell'Aquila
\hfill\break\indent 
67100 Coppito, L'Aquila, Italy
}
\email{riccardo.aragona@univaq.it}

\author{Norberto Gavioli}
\address{\noindent Norberto Gavioli \hfill\break\indent 
	DISIM, Universit\`a dell'Aquila
	\hfill\break\indent 
	67100 Coppito, L'Aquila, Italy
}
\email{norberto.gavioli@univaq.it}

\author{Giuseppe Nozzi}
\address{\noindent Giuseppe Nozzi \hfill\break\indent 
	DISIM, Universit\`a dell'Aquila
	\hfill\break\indent 
	67100 Coppito, L'Aquila, Italy
}
\email{giuseppe.nozzi@graduate.univaq.it}

\begin{document}

\thanks{All  the   authors  are   members  of   INdAM-GNSAGA
  (Italy).} 
\begin{abstract}
In this paper we study the $R$-braces $(M,+,\circ)$ such that $M\cdot M$ is cyclic, where $R$ is the ring of $p$-adic and $\cdot$ is the product of the radical $R$-algebra associated to $M$. In particular, we give a classification up to isomorphism in the torsion-free case and up to isoclinism in the torsion case. More precisely, the isomorphism classes and the isoclinism classes of such radical algebras are in  correspondence with particular equivalence classes of the bilinear forms defined starting from the products of the algebras.
 
\end{abstract}

\noindent

\maketitle
\thispagestyle{empty}

\section{Introduction}
Braces were introduced by Rump in~\cite{RUMP2007153}, as a generalization of Jacobson radical rings, in order to use ring-theoretical and group-theoretical techniques as a tool in the study of non-degenerate involutive set-theoretical solutions of the Yang–Baxter equation.

Recently, different approaches to the study of braces were exploited, considering several connections between these algebraic structures, solutions of the Yang-Baxter equation (see, e.g.~\cite{BCJ,RUMP2007153}), the Hopf-Galois structures (see, e.g.~\cite{Hopf3,Hopf1}), 
and regular subgroups (see, e.g.~\cite{MR3465351,CDVS}). In particular in this paper we address this latter approach, considering the definition of $R$-braces given by Del Corso in~\cite{MR4648557}. In recent years another interesting application of braces to cryptanalysis of block ciphers has been investigated (see e.g.~\cite{civino2,civino1,fedelecritto}).

In this work we call brace a triple $(G,+,\circ)$, where $(G,+)$  and $(G,\circ)$ are both abelian groups such that the following conditions are satisfied
\begin{align*}
     x \circ (y + z) + x = x\circ y + x\circ z\quad\text{and}\quad  (x+y)\circ z +z = x\circ z + y\circ z,
\end{align*}
for every $x,y,z\in G$. A brace $(G,+,\circ)$ is called bi-brace if $(G,\circ,+)$ is also a brace.  It is well known that for an abelian group $(G,+)$, there is a one to one correspondence between bi-braces $(G,+,\circ)$, commutative radical rings $(G,+,\cdot)$ (see e.g. \cite{RUMP2007153}), where
\begin{equation}\label{ciserve}
    x\circ y= x+y+x\cdot y,
\end{equation}
and regular subgroups of $\mathrm{Sym}(G)$ contained in $\mathrm{Hol}(G)$ normalized by $G$ (see e.g. \cite{MR3465351,MR4130907}).

In \cite{CDVS,MR3647970} it has been proved that regular subgroups of the holomorph are in bijective correspondence with the functions, called \emph{gamma functions},
\begin{equation*}
\begin{array}{rcl}
    \gamma:G&\longrightarrow &\Aut(G)\\
    x&\longmapsto & \gamma_x: G \rightarrow G
\end{array}
\end{equation*}
such that $\gamma_{x + \gamma_x(y)} = \gamma_x\gamma_y$.

A brace $(G,+,\circ)$ is called $R$-brace if  $R$ is a commutative ring, the group $(G,+)$ has a structure of $R$-module and $\gamma(G)\subseteq  \Aut_R(G)$. The case when $R$ is a field has been studied also e.g. in \cite{catino4,smok2}.


In this paper, given $p$ an odd prime, we study $R$-braces $M$, when $R$ is  the ring $\mathbb{Z}_p$ of $p$-adic integers, and $M\cdot M$ is a cyclic $R$-submodule. Our aim is to classify such $R$-braces, up to isomorphisms in the torsion-free case and up to isoclinism in the torsion case. A similar classification has already been made in~\cite{nozzi,fedelecritto} for $R=\mathbb{F}_q$, respectively for $q=p^k$, with $p$ odd, and for $q=2$. 

Let $M=\oplus_{i=1}^n \mathbb{Z}_p$ be a free $\mathbb{Z}_p$-module of rank $n$. Let us consider a commutative, torsion-free\ and $3$-nilpotent $\mathbb{Z}_p$-algebra $(M,+,\cdot)$ such that  the $\mathbb{Z}_p$-submodule $M\cdot M$ is cyclic. Let us consider the symmetric bilinear form $b\colon M\times M\to \mathbb{Z}_p$ induced by the product of the algebra $M$. We  show that the isomorphism classes of the algebras $M$ are determined by the congruence classes of these bilinear forms. Then, starting from a bilinear form over $M$, we define a bilinear form, and so a  product, on a torsion submodule of $M$, and we classify the isoclinism classes of the corresponding $\mathbb{Z}_p$-algebras. It is worth to underline that the already mentioned application of $R$-braces to the cryptanalysis of block ciphers has been  investigated in~\cite{civino2,civino1,fedelecritto} emphasizing the case when $M \cdot M$ has rank~$1$.

After the necessary background on braces and lattices over the ring of $p$-adic integers, respectively presented in Section~\ref{sec:1} and Section~\ref{sec:2}, in Section~\ref{sec:3} we deal with $R$-braces, with $R$ any commutative ring. In Section~\ref{sec:p}, we consider the torsion-free $\mathbb{Z}_p$-braces. In particular, we study the congruence classes of symmetric bilinear forms by way of the so-called \emph{Jordan splitting} for lattices over valuation rings (for more details, see e.g. \cite{MR1662447,MR0152507}). In Section~\ref{sec:p+1}, we consider torsion $\mathbb{Z}_p$-lattices 
with symmetric bilinear form $b$. 
Since $\mathbb{Z}_p=\varprojlim_{k\geq 1} \mathbb{Z}/p^k\mathbb{Z}$, it is always possible to classify the bilinear forms $b$ by lifting $M$ to a free $\mathbb{Z}_p$-lattice 
and using the classification already obtained in that context.

In this way, the correspondence between commutative radical rings and bi-braces allow us, in the case when $R$ is $\mathbb{Z}_p$,  to give a complete classification, up to isomorphism in the torsion-free case and up to isoclinism in the torsion case, of $R$-braces $(M,+,\circ)$ with $M \cdot M$ is cyclic.

\section{Preliminaries on Braces}\label{sec:1}

In this section we give the basic notions on \emph{braces}, introduced by Rump~\cite{RUMP2007153}, and the definition of \emph{$R$-braces}, introduced by Del Corso in~\cite{MR4648557}.
\begin{definition}

    A \emph{brace} is a set $G$ with two binary operations, $+$ and $\circ$, such that $(G, +)$ is an abelian group, $(G, \circ)$ is a group, and  any $x,y,z\in G$ satisfy the following conditions
    \begin{align*}
     x \circ (y + z) + x = x\circ y + x\circ z\quad\text{and}\quad  (x+y)\circ z +z = x\circ z + y\circ z,
\end{align*}
\end{definition}

We note that braces were first introduced by Rump in~\cite{RUMP2007153} as a generalization of radical rings. In fact, if $(G,+,\circ)$ is a brace, it is possible to define a radical ring $(G,+,\cdot)$ by setting $
    x\cdot y\defeq x\circ y -x-y \quad\text{for } x,y\in G.
$

It is possible to reinterpret the definition of left brace in terms of regular subgroups of the holomorph of $G$.
 Let $\sigma:G\longmapsto \mathrm{Sym}(G)$ be the right regular representation of $G$, such that $g\sigma_x=g+x$. Recall that the holomorph of $G$ is defined by
 $$
 \mathrm{Hol}(G)=\Aut(G)\ltimes \sigma(G),
 $$ 
and it is isomorphic to the normaliser $N_{\mathrm{Sym}(G)}(\sigma(G))$ of $\sigma(G)$ in $\mathrm{Sym}(G)$.
 
Let $T_\circ$ be a regular permutation subgroup of $\mathrm{Hol}(G)$. We can give a labelling of the elements of $T_\circ$ by the elements of $G$ via $x\mapsto \tau_x$, where $\tau_x$ is the unique element in $T_\circ$ sending $1_G$ to $x$. Since $T_\circ\leq\mathrm{Hol}(G)$ then each element $\tau_x\in T$ can be written uniquely as a product of an element $\gamma_x\in\Aut(G)$ and an element $\sigma_x\in\sigma(G)$. The function $\gamma: G \longmapsto \Aut(G)$, sending $x\in G$ to $\gamma_x$, is called \emph{gamma function} (for more details see e.g., \cite{CDVS,MR4648557,MR3647970}). 
We can define the regular subgroup $(G,\circ)$ of $\mathrm{Hol}(G)$ by the operation
\begin{equation}
    x\circ y\defeq y\tau_x\quad \text{for each }x,y\in G,
\end{equation}
and obtaining that $(G,\circ)$ and $T_\circ$ are isomorphic. In particular, we have a one to one correspondence between braces with additive group $G$ and regular  subgroups of $\mathrm{Hol}(G)$. For more details, see~\cite[Proposition 2.3]{MR3465351}.

\begin{definition}
A brace $(G,+,\circ)$ is called \emph{bi-brace} if $(G,\circ,+)$ is also a brace.
\end{definition}

The same construction as above gives rise to a bijection between bi-braces $(G,+,\circ)$ and regular subgroups of $\mathrm{Hol}(G)$ normalized by $\sigma(G)$ (see~\cite[Theorem 3.1]{MR4130907}).\medskip

We conclude this section giving the definition of $R$-brace, introduced by Del Corso in~\cite{MR4648557}, and a useful result.
\begin{definition}
    Let $(G,+, \circ)$ be a brace with associated gamma function $\gamma\colon G \to \Aut(G)$. Assume that $(G,+)$ has a structure of left (right) module over some ring $R$. We say that $(G,+, \circ)$ is a left (right) \emph{module brace over $R$}, or simply $R$-brace, if $\gamma(G)\subseteq \Aut_R (G)$.
\end{definition}


Throughout this paper we will identify braces,  regular subgroups of the holomorph and radical rings, without further mention, unless necessary.

\section{Preliminaries on \texorpdfstring{$\mathbb{Z}_p$-}{Zp-}lattice}\label{sec:2} 
Let $p$ be an odd prime, $\mathbb{Z}_p$ be the ring of $p$-adic integers and $\mathbb{Q}_p$ be its field of fractions. Let $V$ be an $n$-dimensional $\mathbb{Q}_p$-vector space and let $q$ be a quadratic form over $V$. If $b$ is the symmetric bilinear form associated to $q$ then 
\begin{equation*}
    b(x,y)=\frac{1}{2}(q(x+y)-q(x)-q(y)).
\end{equation*}

\begin{definition}
  A subset $M$ of $V$ is called a \emph{$\mathbb{Z}_p$-lattice} on $V$ if $M$ is a finitely generated $\mathbb{Z}_p$-module.
\end{definition}
It is well known that $\mathbb{Q}_pM=V$ and $M$ is isomorphic to a free $\mathbb{Z}_p$-module of rank $n$ whenever $M$ is $\mathbb{Z}_p$-lattice.
From now on we will call lattice a couple $(M,b)$ and we shall omit $b$ unless necessary.\smallskip

A set of vectors is called \emph{basis} for $M$ if it is a basis for $M$ as a $\mathbb{Z}_p$-module, i.e.\ if it is a basis for $V$ and spans $M$ over $\mathbb{Z}_p$, i.e.\ $\lbrace x_1,\dots,x_n\rbrace$ is a basis for $M$ if and only if it is a basis for $V$ such that
    \begin{equation*}
        M=\mathbb{Z}_p x_1+\dots+\mathbb{Z}_p x_n.
    \end{equation*}

A $n\times n$ matrix $T=(t_{i,j})\in \mathbb{Q}_p^{n\times n}$ is said to be \emph{integral} if each of its entries is in $\mathbb{Z}_p$ and it is called \emph{unimodular} if it is integral and $\mathrm{det}(T)$ is a unit in $\mathbb{Z}_p$.

Let $M$ be a lattice on $V$ and let $A=(b(x_i,x_j))$ and $A'=(b(x_i',x_j'))$ be the matrices associated to $b$ with respect to the basis $\lbrace x_1,\dots,x_n\rbrace$ and $\lbrace x_1',\dots,x_n'\rbrace$ of $M$ then
\begin{equation*}
    A'=T A T^{tr}.
\end{equation*}
for some unimodular matrix $T$.

 As a consequence we get that
\begin{equation*}
    \mathrm{det}(A')=\epsilon^2 \mathrm{det}(A)
\end{equation*}
for some unit $\epsilon$ in $\mathbb{Z}_p$. Hence the canonical image $d(M)$ of $\mathrm{det}(A)$ in $\mathbb{Q}_p^\times /\mathbb{Z}_p^{\times 2}$, called \emph{discriminant} of $M$, is independent of the basis chosen for $M$. 
\begin{definition}
    A lattice $M$ is said to be \emph{unimodular} if the matrix associated to $b$ is unimodular.
\end{definition}

\begin{definition}
We define the \emph{radical} of a lattice $M$ in $V$ to be the sublattice 
\begin{equation*}
    \mathrm{Rad}(M)=\lbrace x\in M \mid b(x,M)=0\rbrace
\end{equation*}
and $M$ is called \emph{regular} if $\mathrm{Rad}(M)=0$.
\end{definition}

We will say that $N$ is the orthogonal sum of two lattices $K$ and $L$, and we will write \(N=K\perp L\), if $N=K\oplus L$  and $N$ is endowed with the bilinear form $b$ which coincides with the forms of $K$ and $L$ when restricted over these submodules and $b(K,L)=0$.
\begin{proposition}[Radical splitting]\label{proposition:splitting}\cite{MR0152507}
Let $M$ be a lattice on $V$. There exists a regular sublattice $K$ such that $M=K\perp \mathrm{Rad}(M)$.
\end{proposition}

We recall that a $\mathbb{Z}_p$-submodule  \(I\) of $\mathbb{Q}_p$ is called a \emph{fractional ideal} if there exists $\lambda \in \mathbb{Q}_p$ such that $\lambda I\subseteq \mathbb{Z}_p$. Notice that $\lambda I$ is a principal ideal in $\mathbb{Z}_p$ so that $I$ is cyclic as a $\mathbb{Z}_p$-module. It is easy to see that every non-trivial finitely generated $\mathbb{Z}_p$-submodule of $\mathbb{Q}_p$ is a fractional ideal.

\begin{definition}
The \emph{scale} of a lattice $M$, denoted by $\mathfrak{s}(M)$, is the fractional ideal generated by the subset $b(M,M)$ of $\mathbb{Q}_p$.
\end{definition}

An isomorphism between two lattices is an isometry if it preserves the bilinear form up to an invertible element in $\mathbb{Z}_p$. We will write \(L\cong M\) to mean \(L\) and \(M\) are isometric.


\begin{theorem}\label{theorem:unimodular}\cite{MR0152507}\label{isocla}
If $M$ is unimodular then 
\begin{equation}
    M\cong  \mathbb{Z}_p\perp\dots\perp \mathbb{Z}_p\perp \mathrm{d}(M)\mathbb{Z}_p
\end{equation}
where $\mathrm{d}(M)$ is the discriminant of $M$.
\end{theorem}
As a conquence we get that the isometry class of an unimodular lattice is completely determined by its rank and its discriminant.\smallskip

Now we state the well known Jordan splitting  theorem for lattices 
(see \cite{MR1662447} for a proof). 
\begin{theorem}[Conway and Sloane]\label{corollary:corollaryJordan}
    Up to isometries a  regular lattice $(M,b)$ can be uniquely decomposed as 
      \begin{equation}
        (M,b)=\perp_{i=1}^\infty  (M_i,b_i)
%
    \end{equation}
where     \(M_i\cong \mathbb{Z}_p^{n_i}\), with \(n_i> 0\) for finitely many \(i\)'s, and \(b_i\) is the bilinear form whose, possibly zero, \(n_i \times n_i\) associated matrix is 
\[J_i (\epsilon)=p^{i}\begin{pmatrix}
    \mathds{1}&0\\
    0&\varepsilon
\end{pmatrix}\] 
and $\varepsilon$ is either $1$ or a pre-selected non-square of $\mathbb{Z}_p$.
\end{theorem}
With the same notation of the previous theorem, we define $b=(b_i)_{i=1}^\infty$ to be the Jordan decomposition of the bilinear form $b$.

We now introduce an equivalence relation between lattices. 
\begin{definition}\label{pteq}
    Two lattices $(M,b_M)$ and $(N,b_N)$ are said to be $p^t$-\emph{equivalent} and we will write $$(M,b_M)\sim_{p^t}(N,b_N),$$ if there exists $\alpha\in \mathbb{Z}_p^{\times}$ and 
    two lattices $(P_1,g_1)$ and $(P_2,g_2)$ with $\mathfrak{s}(P_1)=\mathfrak{s}(P_2)=(p^t)$, such that 
\begin{equation*}
    (M,b_M)\perp(P_1, g_1)\cong (N,\alpha b_N)\perp (P_2, g_2).
\end{equation*}
\end{definition}
\begin{remark}\label{ptremark}
    We notice that $M\sim_{p^t}M'$ if and only if the bilinear forms $b$ and $b'$, respectively of $M$ and $M'$, have proportional Jordan decomposition modulo $p^t$. In other words there exists $\alpha\in \mathbb{Z}_p^{\times}$ such that $$(b_1,\ldots, b_{t-1})\cong\alpha (b'_1,\ldots, b'_{t-1}).$$With an abuse of language we will write $\rk(b_{l})$ to denote the rank $n_l$ of the lattice $(M_{l},b_{l})$. 

\end{remark}


\section{Module braces over a commutative ring}\label{sec:3} 
Let $n$ be a positive integer and $R$ be a commutative ring. Let us consider
\begin{equation}\label{equationM}
    M=\bigoplus_{i=1}^n R
\end{equation} 
the free module of rank $n$ over $R$ and $\lbrace e_1,\dots,e_n\rbrace$ the canonical basis for $M$. Let us denote by $\Aut_{R}(M)$ the automorphism group of the $R$-module $M$ and by 
\begin{equation}\label{eq:affine}
 \mathrm{Aff}(M)=\Aut_{R}(M)\ltimes M
\end{equation} 
the affine group of $M$. The automorphism group of $M$ is
\begin{equation*}
    \Aut_{R}(M)= \mathrm{GL}_{R}(M),
\end{equation*}
which is to the group of unimodular $n\times n$ matrices over~$R$.

Let $T_+$ be the translation group of $(M,+)$, for \(a\in M\) we denote by $\sigma_a$ the translation sending $0$ to $a$.\ Notice that $+$ operation on $M$ can be defined by $a+b\defeq a\sigma_b$, for $a,b\in M$,  so that $(M,+)\cong T_+$ and 
$$
\mathrm{Aff}(M,+)= \mathrm{Hol}(T_+)\cong N_{\mathrm{Sym(M)}}(T_+).
$$

If $T_\circ$ is a regular, abelian subgroup of $\mathrm{Aff}(M)$, then
\begin{equation*}
    T_\circ=\lbrace \tau_a \mid a\in M\rbrace
\end{equation*}
where $\tau_a$ is the unique permutation of $T_\circ$ sending $0$ to $a$. By Equation \eqref{eq:affine}, each element in $T_\circ < \mathrm{Aff}(M)$ can be written 
\begin{equation}\label{dectau}
\tau_a=\gamma_a\sigma_a, \text{ with } \gamma_a\in \mathrm{GL}_{R}(M). 
\end{equation}
Moreover, it is possible to define a new operation \(\circ\) on \(M\) by 
\begin{equation}\label{eq:circoperation}
    a\circ b\defeq a\tau_b
\end{equation}
endowing \(M\) of the structure of a regular abelian group such that $(M,\circ)\cong T_\circ$.


For any $a \in M$, let us denote by $\delta_a$ the endomorphism $\gamma_a-\mathbb{1}_M$ of $(M,+)$. Using the map $\delta_a$ it is possible to define a product on $M$ given by
\begin{equation}\label{eq:prodotto}
    a\cdot b\defeq a\delta_b\quad \text{for all }a,b\in M.
\end{equation}
In the rest of the paper by algebra over a ring \(R\) we shall mean a commutative \(R\)-algebra with finite rank as an \(R\)-module.

\begin{theorem}[\cite{RUMP2007153}]\label{thm:radalg}
   The triple $(M,+,\cdot)$ is a commutative, associative $R$-algebra such that the resulting ring is radical.
\end{theorem}

 \begin{lemma}[\cite{CDVS}]\label{lemma:lemma3}
    For all $a,b\in M$
    \begin{equation*}
        [\sigma_a,\tau_b]=\sigma_{a\cdot b}
    \end{equation*}
    \end{lemma}
   In order to give a bi-brace structure over $(M,+,\circ)$ (see Section~\ref{sec:1}), throughout the paper we will assume the following assumption. 
    \begin{assumption}\label{assum:1}
$T_+<\mathrm{Aff}(M,\circ)\cong N_{\mathrm{Sym(M)}}(T_\circ)$.
\end{assumption}

By Lemma \ref{lemma:lemma3} we get that $T_+$ normalises $T_\circ$ if and only if $\sigma_{a\cdot b}\in T_\circ$ for all $a,b\in M$. Indeed if $T_+$ normalises $T_\circ$, then we get $T_{\circ}^{\sigma{_a}}=T_\circ$ for all $\sigma_a\in T_+$, and
\begin{equation*}
\sigma_{a\cdot b}=\sigma_a^{-1}\tau_b^{-1}\sigma_a\tau_b\in T_\circ.
\end{equation*}
Vice-versa, if $\sigma_{a\cdot b}\in T_\circ$ for all $a,b\in M$, then
    \begin{equation*}
        T_\circ \ni \sigma_{a\cdot b}\tau_b^{-1}=\sigma_a^{-1}\tau_b^{-1}\sigma_a.
    \end{equation*}
    
    Moreover, $\sigma_{a\cdot b}\in T_\circ$, for all $a,b\in M$, if and only if $a\cdot b\cdot c=0$, for all $a,b,c\in M$. Indeed if $\sigma_{a\cdot b}\in T_\circ$ then
        \begin{align*}
           \sigma_{a\cdot b}\in T_\circ\cap T_+&=\lbrace \sigma_x\mid \tau_x=\sigma_x\rbrace\\
           &=\lbrace \sigma_x\mid \gamma_x\sigma_x=\sigma_x\rbrace\\
           &=\lbrace \sigma_x\mid \gamma_x-\mathbb{1}_V=0\rbrace\\
           &=\lbrace \sigma_x\mid \delta_x=0\rbrace\\
           &=\lbrace \sigma_x\mid x\in \ker(\delta)\rbrace.
        \end{align*}
Thus $\sigma_{a\cdot b}\in T_\circ$ if and only if $\delta_{a\cdot b}=0$ for all $a,b\in V$, which, by \eqref{eq:prodotto}, is equivalent to $a\cdot b\cdot c=0$ for all $a,b,c\in M$. 

We conclude  
\begin{lemma}\label{lemnil}
The translation group $T_+$ of $(M,+)$ normalises $T_\circ$ if and only if $(M,+,\cdot)$ is $3$-nilpotent, or, in other words, $M\cdot M\cdot M=0$.
\end{lemma}

Let us denote by $\Ann(M)$ the annihilator of the algebra $(M,+,\cdot)$. Since $M$ is $3$-nilpotent, then
\begin{equation}\label{MM in Ann}
    M\cdot M\subseteq \Ann(M).
\end{equation}

We conclude this section by defining isoclinism of \(R\)-algebras.

\begin{definition}\label{def:isoclinism}
    We shall say that two \(R\)-algebras \(A\) and \(B\) are isoclinic if there exist two isomorphisms
    \begin{equation*}
        \psi\colon A/\Ann(A) \to B/\Ann(B)\quad\text{and} \quad
        \phi \colon A\cdot A \to B\cdot B
    \end{equation*}
    such that the following diagram commutes
    \[
  \begin{tikzcd}
    A/\Ann(A)\times A/\Ann(A) \arrow{r}{\psi \times \psi} \arrow{d}{(\ \cdot \ )}& B/\Ann(B)\times B/\Ann(B)\arrow{d}{(\ \cdot \ )} \\
      A\cdot A \arrow{r}{\phi} &B\cdot B
  \end{tikzcd}
\]
Isoclinism is an equivalence relation among \(R\)-algebras. We say that an algebra \(S\) is \emph{stem}  if \(Ann(S)\subseteq S\cdot S\). For two isoclinic algebras $A$ and $B$ we shall write $A\sim B$.
\end{definition}
We notice that the relation of isoclinism of $R$-algebras is a particular case of the isoclinism of skew braces defined in \cite{MR4698318}. In according to P.Hall \cite{Hallpgroup}, we give the following characterization.
\begin{lemma}
    Let $A$ be an $R$-algebra, $B\subseteq A$ a subalgebra and $I\subseteq \Ann(A)$ an ideal of $A$ such that $I\cap (A\cdot A)= \{0\}$. Then
    \begin{itemize}
        \item $A\sim A\oplus T$, where $T$ is any algebra such that $T\cdot T=0$;
        \item $A\sim A/I$;
        \item $A\sim B$ if and only if $A=B+\Ann(A)$.
        
    \end{itemize}
\end{lemma}
For a proof of the following result see \cite[Theorem 2.18]{MR4698318}.
\begin{theorem}
    Every $R$-algebra is isoclinic to a stem algebra.
\end{theorem}
\begin{definition}\label{stemalgebra}
Let $M$ and $N$ be two $R$-modules and let $b\colon M\times M\to N$ be a $R$-bilinear form such that $N=\langle b(M,M)\rangle$. We define the $R$-algebra $$\Stem(b)=\left(M/\Ann(M)\right)\oplus N$$ endowed with the product given by $(m_1,n_1)\cdot (m_2,n_2)=(0,b(m_1,m_2))$.
\end{definition}
We point out that $\Stem(b)$ is a $3$-nilpotent algebra. In case of $(M,+,\cdot)$ is a $3$-nilpotent algebra then $N=M\cdot M$ and $b$ is the bilinear form associated to the product.

\section{Torsion-free module braces over \texorpdfstring{$\mathbb{Z}_p$}{Zp}}\label{sec:p}
In this section we deal with the torsion-free case giving a classification up to isomorphisms. Using the same notation of the previous section we consider $p$ an odd prime and 
$$M=\bigoplus_{i=1}^n \mathbb{Z}_p$$
satisfyning Assumption \ref{assum:1}. Furthermore we introduce the following assumption which allows us to study the product in \(M\) as a standard symmetric bilinear form. 
\begin{assumption}\label{assumption: MM uno}
    $M\cdot M$ is a cyclic \(\mathbb{Z}_p\)-module. 
\end{assumption}
Since $\mathbb{Z}_p$ is a PID, notice that if $M \cdot M$ is a cyclic submodule of the free $\mathbb{Z}_p$-module $M$, then $M\cdot M \cong\mathbb{Z}_p$ as a $\mathbb{Z}_p$-module.
Let us define the symmetric bilinear form  
\begin{equation}\label{eq:bilinearform}
    b:M\times M\longmapsto \Ann(M) ,\ (a,b)\mapsto a\cdot b.
\end{equation}
which turns \(M\) into a lattice (see Section~\ref{sec:2}).
Notice that the radical $\mathrm{Rad}(M)=\lbrace a\in M:\ b(a,M)=0\rbrace$ of the bilinear form $b$  coincides with $\Ann(M)$. Moreover, $\mathrm{Rad}(M)$ is a submodule of the free module $M$,  so it is also free, say of rank $d$. By Proposition \ref{proposition:splitting}, we can find a basis for $M$ such that $\mathrm{Rad}(M)$ is spanned by the last $d$ vectors of that basis, i.e.
\begin{equation}\label{eq:spansocle}
   \mathrm{Rad}(M)=\mathrm{span}_{\mathbb{Z}_p}\lbrace e_{m+1},\dots,e_n\rbrace,\quad m\defeq n-d.
\end{equation}

The next result relates the matrix associated to the bilinear form $b$ with the linear maps described in~\eqref{dectau}.
\begin{lemma}\label{lemma:formagamma}
There exists a $m\times d$ matrix $\Theta_i$ with entries in $\mathbb{Z}_p$ such that $\gamma_{e_i}$ has the form
        \begin{equation*}
            \gamma_{e_i}=
            \begin{pmatrix}
                \mathbb{1}_{m} & \Theta_i\\
               0_{d\times m} & \mathbb{1}_{d} 
            \end{pmatrix}\quad \text{for } i=1,\dots,m.
        \end{equation*}
Moreover $\Theta_i=0_{m\times d}$ for $i=m+1,\dots,n$.
\end{lemma}
\begin{proof}
    Notice that $a\in \Ann(V)$ if and only if $a\circ b=a+b$ for all $b\in M$. It follows that $e_j\circ e_i=e_j+e_i$ for $j=m+1,\dots,n$ and $i=1,\dots, n$. In other words
\begin{equation}\label{1}
    e_j\circ e_i=e_j\gamma_{e_i}+e_i=e_j+e_i.
\end{equation}
On the other hand for $j=1,\dots,m$ and $i=1,\dots, m$
\begin{equation}\label{2}
   e_j\circ e_i=e_j\gamma_{e_i}+e_i=e_j+e_i+e_j\cdot e_i
\end{equation}
where $e_j\cdot e_i\in \Ann(V)=\mathrm{span}\{e_{m+1},\dots,e_n\}$. Thus, by Equation \eqref{1} and \eqref{2} we obtain 
\begin{equation*}
    \gamma_{e_i}=\begin{pmatrix}
        \mathbb{1}_m& \Theta_{e_i}\\
        0&\mathbb{1}_d
    \end{pmatrix}\ \text{for }i=1,\dots,m
\end{equation*}
and $\gamma_{e_i}=\mathbb{1}_n$ for $i=m+1,\dots,n$. Moreover, denoting by $\Theta_{i,j}$ the $j$-$th$ row of the matrix $\Theta_i$, we get that $e_i\cdot e_j=(\underbrace{0,\dots,0}_m,\Theta_{i,j})$. Finally
\begin{equation*}
\gamma_a=\begin{pmatrix}
    \mathbb{1}_m& a_1\Theta_{e_1}+\dots+a_m\Theta_{e_m}\\
    0&\mathbb{1}_d
\end{pmatrix}
\end{equation*}
for every $a=a_1e_1+\dots+a_ne_n\in M$.\qedhere
\end{proof}
\begin{remark}\label{Ann=MM}
By Equation \eqref{MM in Ann}, $\Ann(M)=H\oplus K$, where \(K\) is a cyclic \( \mathbb{Z}_p \)-module containing \(M\cdot M\).
 Up to consider the quotient algebra $(M/H,+,\cdot)$, we can assume $M\cdot M=c\cdot \Ann(M)$, with \(c\in \mathbb{Z}_p \) and $\Ann(M)=\mathrm{span}\lbrace e_n\rbrace $ is the submodule spanned by the last vector of the basis. Therefore, without lost of generality, from now on we will consider $m=n-1$,  $M\cdot M=c\cdot \mathrm{span}\lbrace e_n \rbrace$ and $M=N\oplus \Ann(M)$ where $N=\mathrm{span}\lbrace e_1,\dots,e_{n-1}\rbrace$.
\end{remark}
\begin{definition}\label{defmat1}
We shall say that the matrix $\Theta\defeq[\Theta_1\ldots \Theta_{n-1}]$ is the \emph{defining matrix} in the given basis of the commutative, torsion-free and $3$-nilpotent $\mathbb{Z}_p$-algebra $(M,+,\cdot)$ with $\Ann(M)$ cyclic.
\end{definition}
We notice that the matrix $\Theta=[\Theta_1\ldots \Theta_{n-1}]$ is the matrix associated to the symmetric bilinear form $b$ restricted on the submodule $N$ defined in Remark \ref{Ann=MM}.
\begin{lemma}
 The matrix $\Theta$ is symmetric of maximal rank.
\end{lemma}
\begin{proof}
     Since the algebra $(M,+,\cdot)$ is commutative, $e_i\cdot e_j=e_j\cdot e_i$ for  every $i,j\in \lbrace 1,\dots,n\rbrace$, and by Lemma \ref{lemma:formagamma}
\begin{equation*}
    e_i\cdot e_j=(0,\dots,0,\Theta_{i,j})=(0,\dots,0,\Theta_{j,i})=e_j\cdot e_i.
\end{equation*}
It follows that the matrix $\Theta$ is symmetric.\\
Let us suppose that a non trivial $\mathbb{Z}_p$-linear combination of the column vectors $\Theta_1,\dots,\Theta_{n-1}$ is the null vector, i.e.
\begin{equation*}
           \displaystyle\sum_{i=1}^{n-1} a_i\Theta_i =0\quad \ a_i\in \mathbb{Z}_p\ \text{and } a_s\neq 0\text{ for some }s\in \lbrace 1,\dots,n-1\rbrace.
       \end{equation*}

        Then 
        \begin{equation*}
            \gamma_{a_1 e_1+\dots+a_{n-1} e_{n-1}}=
            \begin{pmatrix}
                \mathbb{1}_{n-1} & a_1 \Theta_1+\dots+a_{n-1} \Theta_{n-1}\\
                0 & 1
            \end{pmatrix}
            =\mathbb{1}_n
        \end{equation*}
This means that $a_1 e_1+\dots+a_{n-1} e_{n-1}\in\Ann(M)\cap N=\{0\}$.
\end{proof}

Conversely we have the following result.

\begin{theorem}\label{thm:theoremtheta}
Any $(n-1)\times (n-1)$ symmetric matrix with coefficients in $\mathbb{Z}_p$  of maximal rank 
is the defining matrix (with respect to a suitable basis) of a commutative, torsion-free and $3$-nilpotent $\mathbb{Z}_p$-algebra with $\Ann(M)$ cyclic.
\end{theorem}

\begin{proof} 
Let $\Theta=[\Theta_1,\ldots,\Theta_{n-1}]$ be a  $(n-1)\times (n-1)$ symmetric matrix with coefficients in $\mathbb{Z}_p$ of maximal rank  and let $(M,+)$ be a free $\mathbb{Z}_p$-module of rank $n$. For $a=a_1 e_1+\cdots +a_n e_n\in M$, let us  define the map $\tau_a=\gamma_a\sigma_a$ where
\begin{equation*}
    \gamma_a=\begin{pmatrix}
        \mathbb{1}_{n-1}&a_1\Theta_1+\cdots +a_{n-1}\Theta_{n-1}\\
        0 &1
    \end{pmatrix},
\end{equation*}
and $\sigma_a$ is a translation of $(M,+)$. 
Set $T_\circ=\lbrace \tau_a:\  a\in M\rbrace < \mathrm{Aff}(M)$, where $\circ$ is the operation induced by $T_\circ$ on $M$ as in Equation \eqref{eq:circoperation}.

\noindent Let $\cdot$ be as in \eqref{eq:prodotto}, in order to prove that $(M,+,\cdot)$ is a commutative and $3$-nilpotent algebra, by Theorem \ref{thm:radalg} it is enough to show that $T_\circ$ is an abelian regular subgroup of $\mathrm{Aff}(M)$ normalized by $T_{+}$.
\begin{itemize}
    \item $T_\circ$ is a group indeed $\tau_0=\mathbb{1}_M$ is the neutral element and $\circ$ is associative by definition. Notice that $\tau_a\tau_b=\tau_{a\circ b}$, indeed 
    \begin{equation*}
        x\tau_a\tau_b=x\gamma_a\gamma_b+a\gamma_b+b=x\gamma_a\gamma_b+a\circ b
    \end{equation*}
    and since $a\circ b=a+b+a\cdot b$ and $\gamma_{a\cdot b}=\mathbb{1}_n$ 
    \begin{equation*}
        x\tau_{a\circ b}=x\gamma_{a\circ b}+a\circ b=x\gamma_{a+b}+a\circ b=x\gamma_a\gamma_b+a\circ b.
    \end{equation*}
    Thus, to prove that every map of $T_\circ$ admits an inverse, we just have to prove that each $a\in M$ admits an inverse with respect to $\circ$ operation.
    In other words, if $a=(a_1,\dots,a_n)\in M$, then we have to find $b\in M$ such that $a\circ b=0$, i.e $a+b+a\cdot b=a\circ b=0$. By Remark \ref{Ann=MM}, $M\cdot M$ has rank 1 and is spanned by $e_n$, so $b=(-a_1,\dots,-a_{n-1},b_n)$. It remains to find the last component $b_n$, that is
    \begin{align*}
        b_n e_n&=-(a_1,\dots,a_n)\cdot (-a_1,\dots,-a_{n-1},b_n)-a_n e_n\\
        &=((a_1,\dots,a_{n-1})\Theta (a_1,\dots,a_{n-1})^{tr}-a_n)e_n.
    \end{align*}
\item $T_\circ$ is abelian, indeed by definition $a\circ b=b\circ a$ and we have just seen that $\tau_a\tau_b=\tau_{a\circ b}$.
\item $T_\circ$ is regular, i.e. for all $ a,b\in M$ there exists a unique $c\in M$ such that $a\tau_c=b$. For proving this, it is enough to observe that $c=b\tau_a^{-1}$, where the uniqueness of $c$ follows from the uniqueness of the inverse in $T_\circ$.
\item $T_+$ normalises $T_\circ$, i.e., for all $a,b\in M$, $\sigma_a \tau_b\sigma_a^{-1}\in T_\circ$. Indeed, using the equality $\gamma_a\gamma_b=\gamma_{a\gamma_b+b}$, it follows that
\begin{align*}
    x\sigma_a\tau_b\sigma_{-a}&=(x+a) \gamma_b\sigma_b\sigma_{-a}\\
    &=x\gamma_b+a\gamma_b+b-a \\
    &=x\gamma_a\gamma_b\gamma_{-a}+a\gamma_b+b-a\\
    &=x\tau_{a\gamma_b+ b-a}. \qedhere
\end{align*}
\end{itemize}
\end{proof}

Since the algebras appearing in the previous theorem will be often considered in the rest of this section we give a formal definition.  



A straight consequence of Theorem~\ref{thm:theoremtheta} is the following result which gives a relationship between isomorphism classes of commutative, torsion-free and $3$-nilpotent $\mathbb{Z}_p$-algebras with cyclic annihilator and congruence classes of the associated bilinear forms.

\begin{proposition}\label{proposition:isomorphismmatrices}
    Two commutative, torsion-free and $3$-nilpotent $\mathbb{Z}_p$-algebras  $M_1=(M,+,\cdot_1)$ and $~M_2=(M,+,\cdot_2)$ with cyclic annihilators and with defining matrices $\Theta_{M_1}$ and $\Theta_{M_2}$, respectively, are isomorphic if and only if there exists a unimodular $(n-1)\times (n-1)$ matrix $A$ such that
    \begin{equation*}
        A\Theta_{M_1} A^{tr}=\varepsilon\Theta_{M_2}
    \end{equation*}
    where $\varepsilon\in \lbrace 1,q\rbrace$ and $q$ is a non-square element of  $\mathbb{Z}_p$.
\end{proposition}

\begin{corollary}\label{corollary:88}
    If $\Theta_{M_1}$ and $\Theta_{M_2}$ are congruent then the associated algebras $M_1=(M, +, \cdot_1)$ and $M_2=(M, +, \cdot_2)$ are isomorphic.
\end{corollary}
Notice that for the general case as in Remark \ref{Ann=MM}, where $\Ann(M)=H\oplus K$ with \(K\) a cyclic \( \mathbb{Z}_p \)-module containing \(M\cdot M\), we have that the isomorphism classes of $M$ are in one to one correspondence with the isomorphism classes of $M/H$. 
Hence Theorem \ref{thm:theoremtheta} and Proposition \ref{proposition:isomorphismmatrices} allow us to  classify the isomorphism classes of the commutative, tosion-free and $3$-nilpotent $\mathbb{Z}_p$-algebras $M$ with $M\cdot M$ cyclic.

Next two results are direct consequences respectively of Theorem \ref{theorem:unimodular} and Theorem \ref{corollary:corollaryJordan}.
\begin{lemma}
Let $\Theta$ be the defining matrix of $(M,+,\cdot)$. If $\mathrm{det}(\Theta)$ is a unit in $\mathbb{Z}_p$ then 
\begin{equation*}
    \Theta\cong
\left(\begin{array}{cccc}
1&0&0&0\\
0&1&0&\vdots\\
\vdots &0 &\ddots &0\\
0&\cdots&0&d(\Theta)
\end{array}\right)
\end{equation*}
where $d(\Theta)$ is the discriminant of the matrix $\Theta$. \\
\end{lemma}
More in general we have the following result.
\begin{lemma}\label{lemma: corcorcor}
Let $\Theta$ be the $n\times n$ defining matrix of $(M,+,\cdot)$. Then $M=\perp_{j=1}^t M_{i_j}$ and 
\begin{equation*}
    \Theta\cong
\left(\begin{array}{cccc}
J_{i_1}(\varepsilon_{1})&0&0&0\\
0&J_{i_2}(\varepsilon_{2})&0&\vdots\\
\vdots &0 &\ddots &0\\
0&\cdots&0&J_{i_t}(\varepsilon_{t})
\end{array}\right),\; \text{ with } \; J_k(\varepsilon)=p^k\begin{pmatrix}
    \mathds 1 &0\\
    0& \varepsilon
\end{pmatrix},
\end{equation*}
and  $0\leq i_1<\dots<i_t\leq n$.

\end{lemma}

We are now in the position to give a complete classification of the isomorphism classes of $(M,+,\cdot)$.
\begin{theorem}[Unimodular case]\label{theorem:unimodularcase}
Let $(M,+,\cdot)$ be any commutative, torsion-free and $3$-nilpotent $\mathbb{Z}_p$-algebra such that $M$ has rank $n$ with $\Ann(M)$ of rank $d$ and $M\cdot M$ cyclic. If the defining matrix $\Theta$ of $M$ is unimodular then there are two isomorphism classes of $M$ in the case of $n-d$ is even and one isomorphism class in the case of $n-d$ is odd. Moreover, if $n-d$ is even the isomorphism classes are determined by $d(\Theta)$ .
\end{theorem}
\begin{proof}
By what observed after Corollary \ref{corollary:88} we may assume $(M,+,\cdot)$ to be commutative, torsion-free and $3$-nilpotent $\mathbb{Z}_p$-algebras with cyclic annihilator (i.e. $d=1$).
Let us fix $\Theta_{1}=\langle 1\rangle\perp\dots\perp\langle 1\rangle$ and $\Theta_{2}=\langle 1\rangle\perp\dots\langle 1\rangle\perp\langle q\rangle$ where $q$ is a non-square element of $\mathbb{Z}_p$. It is clear that $\Theta_{1}$ and $\Theta_{2}$ can not be congruent since they have different discriminant.\\
Thus, we just have to check when $\Theta_{1}$ and $q\Theta_{2}$ are congruent using Proposition~\ref{proposition:isomorphismmatrices}.
\begin{itemize}
    \item If $n-1$ is even, then  $\mathrm{det}(\Theta_{1})=1$ is a square and $\mathrm{det}(q\Theta_{2})=q^{n}$ is not a square. Then $\Theta_{1}$ and $q\Theta_{2}$ are not congruent.
    \item If $n-1$ is odd, then   $\mathrm{det}(\Theta_{1})=1$ is a square and $\mathrm{det}(q\Theta_{2})=q^{n}$ is a square. Then $\Theta_{1}$ and $q\Theta_{2}$ are congruent.\qedhere
\end{itemize}
\end{proof}
\begin{remark}
 In the general case, the defining matrix $\Theta$ of any commutative, torsion-free and $3$-nilpotent $\mathbb{Z}_p$ algebra $(M,+,\cdot)$ with $M\cdot M$ cyclic is not unimodular and there are infinite isomorphism classes of such algebras. Indeed, by Lemma \ref{lemma: corcorcor}, we can represent $\Theta$ in its Jordan form; counting the isomorphism classes of such algebras is equivalent to count all the possible Jordan splitting of a $n\times n$ symmetric matrix over $\mathbb{Z}_p$. There are infinitely many such splittings since the possible choices for the scales of the Jordan blocks.

 Moreover, let $I$ be an isomorphism class of commutative, torsion-free and $3$-nilpotent algebras with cyclic annihilator and let  $\Theta$ be the defining matrix (represented in its Jordan form) of an algebra in that class. It follows by Proposition \ref{proposition:isomorphismmatrices} that if $\Theta$ has at least one Jordan block of odd dimension then $|I|=2$, otherwise $|I|=1$.

\end{remark}

\section{Isoclinisms of torsion module braces over \texorpdfstring{$\mathbb{Z}_p$}{Zp}}\label{sec:p+1}
In this section we deal with the torsion case, giving a classification up to isoclinisms (Definition \ref{def:isoclinism}). 
We shall denote by $(M,+,\cdot)$ a $3$-nilpotent algebra as defined in Section \ref{sec:3} with $M\cdot M$ cyclic. As above the product induces a symmetric bilinear form \[b\colon M/\Ann(M)\times M/\Ann(M)\to M\cdot M.\]  
Due to the torsion assumption, the module \(M\cdot M\) is a finite and cyclic  $\mathbb{Z}_p$-submodule of $M$ and \(M/\Ann(M)\) is a finite \(\mathbb{Z}_p\)-module. In particular $M\cdot M\cong \mathbb{Z}_p/p^t\mathbb{Z}_p$ for some $t>0$.
We notice that $(M,+,\circ)$ is a $\mathbb{Z}_p$-brace with $\circ$ defined in Equation \eqref{ciserve}.



For any finitely-generated free $\mathbb{Z}_p$-module $\overbar{M}$  having \(M/\Ann(M)\) as a quotient and projection \(\pi\colon \overbar{M} \to M/\Ann(M)\), there exists a bilinear form making \( (\overbar{M} , \bar{b}) \) into a lattice such that the following diagram
    \begin{equation*}\label{diagram}
         \begin{tikzcd}
    (\overbar{M})^2\arrow{d}{\bar{b}} \arrow{r}{\pi}& (M/\mathrm{Ann}(M))^2\arrow{d}{b}\\
      \mathbb{Z}_p \arrow{r}{\Theta} &M\cdot M
  \end{tikzcd}
    \end{equation*}
    commutes, where $\Theta$ is the reduction modulo $p^t$. Any such lattice is said to be a \emph{covering} of the algebra \(M\) and \(\bar b\) is said to be a \emph{lifting} of \(b\). Any two coverings of \(M\) are easily seen to be \(p^t\)-equivalent according to Definition~\ref{pteq}.

    We point out that there are actually infinitely many coverings of \(M\). Notice that  $\bar{b}_1$ and $\bar{b}_2$ are two symmetric bilinear forms on \(\overbar M\) making the previous into a commutative diagram if and only if there exists another bilinear form $\bar{b}_3$ over $\overbar{M}$ such that 
\begin{equation*}
    \bar{b}_1-\alpha \bar{b}_2=p^t \bar{b}_3
\end{equation*}
where $\alpha$ is an invertible element of $\mathbb{Z}_p$.\ We denote by $[b]$ the equivalence class of lattices that are covering of \(M\).
\begin{remark}
  It is actually possible to find infinitely many non-degenerate liftings of $b$. Indeed, let $\bar{b}\colon \overbar{M}\times \overbar{M}\to\mathbb{Z}_p$ be a symmetric lifting of $b$ with associated matrix $\overbar{B}$, then the symmetric lifting with associated matrix $\overbar{B}-p^h\mathds{I}_n$, for $p^h$ not an eigenvalue of $\overbar{B}$ and $h\geq t$, provides a non-degenerate covering in the same class.  
\end{remark}


From Remark \ref{ptremark}, using the same notation, it follows immediately the next result.
\begin{proposition}\label{proposition:bilform}
If  two coverings $(\overbar M,\bar b)$ and $(\overbar N,\bar g)$ with Jordan decompositions modulo $p^t$ respectively $(\bar b_{1},\ldots,\bar b_{t-1})$ and $(\bar g_{1},\ldots,\bar g_{t-1})$ are in $[b]$ for some $b$, then 
\begin{enumerate}
    \item $\rk(\bar b_{i})=\rk(\bar g_{i})$ for each $i=1,\ldots,t-1$, and
    \item $\bar b_{i}=\bar g_{i}$ if $\rk(\bar b_i)=\rk(\bar g_i)$ is even.
\end{enumerate}

\end{proposition}
\begin{proof}
We already observed that two coverings of $M$ are $p^t$-equivalent. It follows that  $(\bar b_{1},\ldots,\bar b_{t-1})\cong\alpha(\bar g_{1},\ldots,\bar g_{t-1})$ for some $\alpha\in \mathbb Z_p^{\times}$ and so $\rk(\bar b_{i})=\rk(\bar g_{i})$ for each $i=1,\ldots,t-1$. Since multiplication by $\alpha$ does not change the discriminant of Jordan blocks of even rank, if $\rk(\bar b_i)=\rk(\bar g_i)$ is even then $\bar b_i=\bar g_i$. 
\end{proof}

We consider now a symmetric bilinear form $\bar{b}\colon\overbar{M}\times \overbar{M}\to \mathbb{Z}_p$ with $p^t$-radical $K=\lbrace x\in \overbar{M}\colon \bar{b}(x,y)\in p^t\mathbb{Z}_p,\ \text{for all } y\in \overbar{M}\rbrace$, ant $t>0$ and we construct an algebra $M$ whose product form $b$ is such that $\bar b\in [b]$. Clearly $K$ is determined by the Jordan splitting of $\bar{b}$ and it is equal to the orthogonal sum of the $h$-Jordan components with $h\geq t$. Let $\Tilde{b}\colon \overbar{M}\times \overbar{M}\to \mathbb{Z}_p/p^t\mathbb{Z}_p$ be the reduction of $\bar b$ modulo $p^t$ and let $M=\Stem(\Tilde{b})$ be as in Definition \ref{stemalgebra}. We have $\overbar{M}/K\cong M/\Ann(M)$ and the following commutative diagram

\begin{equation*}
\begin{tikzcd}
   (\overbar{M})^2\arrow{d}{\bar{b}}\arrow{r}{}\arrow{rd}{\Tilde{b}}&(M/\Ann(M))^2\arrow{d}{b}\\
      \mathbb{Z}_p \arrow{r}{\Theta} &\mathbb{Z}_p/p^t\mathbb{Z}_p
  \end{tikzcd}
\end{equation*}
where $b$ is the symmetric bilinear form induced by the product in the algebra $M$.

\begin{proposition}\label{propeccoqui}
    The non-degenerate lattice $(\overbar{M},\bar b)$ is the covering of any algebra in the isoclinism class of $M=\Stem(\bar{b}\bmod p^t)$.
\end{proposition}
\begin{proof}
    We define $K=\lbrace x\in \overbar{M}\colon \bar{b}(x,y)\in p^t\mathbb{Z}_p,\ \text{for all } y\in \overbar{M}\rbrace$ and we note that $\overbar{M}/K\cong M/\Ann(M)$. 
    Let $N$ be an algebra in the isoclinism class of $M$ with $\psi$ and $\phi$ isomorphisms as in Definition \ref{def:isoclinism}. It is possible to construct the following diagram
    \begin{equation}
         \begin{tikzcd}
    (\overbar{M})^2\arrow{d}{\bar{b}} \arrow{r}{\pi}& (\overbar{M}/K)^2\cong (M/\Ann(M))^2\arrow{d}{b_M}\arrow{r}{\psi\times \psi}& (N/\Ann(N))^2\arrow{d}{b_N}\\
      \mathbb{Z}_p \arrow{r}{\Theta} &\mathbb{Z}_p/p^t\mathbb{Z}_p\cong M\cdot M\arrow{r}{\phi}&N\cdot N
  \end{tikzcd}
    \end{equation}
    where $b_M$ and $b_N$ are the bilinear maps associated to the products of $M$ and $N$ respectively. The statement follows since the previous diagram is commutative.
\end{proof}
We prove now that there are finitely many isoclinism classes of  commutative, $3$-nilpotent, torsion $\mathbb Z_p$-algebras $M$ with $M\cdot M$ cyclic.
In the following statement we set $\binom{a}{b}=0$ if $a$ is not an integer.
\begin{theorem}
There are 
\begin{equation*}
    \displaystyle\sum_{s=1}^{t} \binom{t}{s}2^{s-1}\left(    \binom{n-1}{s-1}+\binom{n/2-1}{s-1}\right) 
\end{equation*}
isoclinism classes of commutative, $3$-nilpotent $\mathbb{Z}_p$ algebras $M$ with torsion of rank $n$ and $M\cdot M\cong \mathbb{Z}_p/p^t\mathbb{Z}_p$.

\end{theorem}
\begin{proof}
We start counting all the possible Jordan splitting modulo $p^t$ of a $\mathbb Z_p$-lattice.
All the possible splittings of $(\overbar M, \bar b)$ in $s$-components can be obtained by choosing $s$ scales in $\binom{t}{s}$ ways. The multiplication by an non-square invertible element $\alpha\in \mathbb Z_p$ changes the discriminant of odd-rank Jordan blocks and  leaves the even-rank ones unchanged. Supposing that at least one of the involved rank is odd, since the discriminant of each block can be chosen
in two ways, we obtain 
    \begin{equation*}
        \displaystyle\sum_{s=1}^{t} \binom{t}{s}2^{s-1} H(n,s)
    \end{equation*}
possible non-equivalent Jordan splittings, where $H(n,s)$ is the number of decompositions of $n=n_1+n_2+\ldots+n_s$ with $n_i>1$ for all $i=1,\ldots,s$ and at least one $n_i$ odd.
We define also $K(n,s)$ the number of decompositions of $n=n_1+n_2+\ldots+n_s$ with $n_i>1$ for all $i=1,\ldots,s$.  
It remains to consider the case when all the ranks $n_i$, and in particular also $n$, are even. The corresponding decomposition of $n$ is determined by doubling a decomposition of $n/2$ in $s$ parts. There are $K(n/2,s)$ such decompositions, each of them carrying $2^s$ non-equivalent splittings. Hence there are in total 
\begin{gather*}
        \displaystyle\sum_{s=1}^{t} \binom{t}{s}2^{s-1} H(n,s)+ \sum_{s=1}^{t} \binom{t}{s}2^{s} K(n/2,s)=\\
        \displaystyle\sum_{s=1}^{t} \binom{t}{s}2^{s-1} K(n,s)+ \sum_{s=1}^{t} \binom{t}{s}2^{s-1} K(n/2,s).
    \end{gather*}
    It is well known that $K(n,s)=\binom{n-1}{s-1}$ concluding the proof.\qedhere

\end{proof}
\bibliographystyle{abbrv}
\bibliography{citation}
\end{document}